\documentclass[a4paper,12pt]{llncs}
\usepackage{graphicx}
\begin{document}
\frontmatter
\mainmatter

\title{Probabilistic Generative Model of Social Network Based on Web Features}

\titlerunning{Probabilistic Generative Model of Social Network ...}

\author{Mahyuddin K. M. Nasution\inst{1,2} \and Shahrul Azman Noah\inst{2}}

\authorrunning{Nasution, M. K. M.}
\tocauthor{Noah, S. A.}
\institute{Department of  Mathematics FMIPA USU, Medan, Indonesia\\
\email{mahyunst@yahoo.com}, \email{mahyuddin@usu.ac.id}
\and
Knowledge Technology Group Research FTSM UKM, Bangi, Malaysia\\
\email{samn@ftsm.ukm.my}}
\maketitle

\begin{abstract}
In this paper, we develop a dynamic framework for the modeling and analysis of social networks to work with web documents. We illustrate the model with features of web, design a form to analyze relationships of attributes as a modality of social structure, and create the optimization of generative model based on Bayes Theorem.
\end{abstract}

\section{Introduction}
Social network describes a group of social entities and the pattern of inter-relationships among them. The concept of social networks is designed to map the relationship of entities among all of them that can be observed, to mark the patterns of ties between entities, to measure social capital: the values obtained by the entities individually or in groups, to present a variety of social structures according to the interests and its implementation, based on different domains or information sources \cite{nasution2010a}. Group discovery has many applications, such as understanding the social structure of organizations or native tribes. In law enforcement this is about organized crimes such as drugs and money laundering \cite{xu2004} or terrorism \cite{krebs2002}, knowing how the perpetrators are connected to one another would assist the effort to disrupt a criminal act or to identify additional suspects. In commerce, viral marketing exploits the relationship between existence of customers and potential customers to increase sales of products and services \cite{richardson2002,kempe2003}. Members of a social network may also take advantage of their connections to get to know others, for instance through web sites facilitating networking or dating among their users \cite{boyd2004}.

Social networks explicitly exhibit relationships (called \emph{ties} in social sciences) among individuals and groups (called \emph{actors}). They have studied social sciences since the 1930s. Social scientists have conducted extensive research on group detection, especially in fields such as anthropology and political science. Recently, statisticians and computer scientists have begun to develop models that specifically discover group memberships \cite{nowicki2001,kubica2002,bhattacharya2004}. There are two models use probability for characterizing information sources as well as image, has become a tool in machine learning research: probabilistic generative and relative models. Therefore, the approaches that addresses the issues of social network generally fall into two categories also, and this paper explores  their relationship in order to obtain a framework to investigate social networks by engaging web features. 

\section{Related Work and Problem}
Technically, an object called the network is a graph. A graph consist of a set of points along with a set of lines connecting pairs of points. Formally, a graph denoted by $G(V,E)$, where $V\ne\emptyset$ is a set of vertices or nodes, $V = \{v_i|i=1,\dots,I\}$ and $E$ is a set of edges connects between pair of vertices, $E = \{e_j|j=1,\dots,J\}$. The  
nodes in a social network refer to actor names such as authors, recipients, researchers, artists, politicians, firms, organizations, or any entity, i.e.,

\begin{definition}
A set of actors $A = \{a_i|i=1,\dots,I\}$ and there is a function $\xi$ such that $\xi : A \stackrel{1:1}{\rightarrow} V$, or $\forall a\in A \exists! v\in V$.
\end{definition}

Each actor plays some role in the social interactions based on his/her background and then achievements obtained in each occasion, or kinds of: new articles and academic publications. In many situations, these are considered as attributes or characteristics, and the characteristic can be defined as follows.

\begin{definition}
Let $Z = \{z_k|k=1,\dots,K\}$ is a set of attributes, and a pair of $\langle A,Z\rangle$ is the instance of actors, where $Z_i$ is subsets of $Z$, $Z_i$ are subsets of attributes of each actor $a_i$, i.e., $\langle a_i,Z_i\rangle$, $i=1,\dots,I$, simply we denotes a set of attributes of actor $a$ as $Z_a$.
\end{definition}

A social network is a network based on the relations between people in their society. Therefore, we can model an approach of other social network. When a computer network connects people, it is a social network \cite{garton1997}. Just as a computer network is a set of machines connected by a set of media (cables or airwaves), a social network also is a set of people connected by a set of social relationships such as friendship, co-working or information exchange such as Web. Information of people in web documents is very different from information of people in database. In any documents, the objects (actors and attributes) can be given literally, like the literal text of \emph{Indonesia}, then all meaning of object based on words represented by the literal objects itself. To realize it, first we define that a word $w$ is the basic unit of discrete data, defined to be an item from a vocabulary indexed by $\{1,\dots,K\}$, where $w_k = 1$ if $k\in K$, and $w_k=$ otherwise. Then, we define some instances related to words.

\begin{definition}
\label{def:doc}
A document is a sequence of $n$ words denoted by $D = \{w_1,\dots,$ $w_N\}$, where $w_n$ is the $n$-th word in the sequence. Size of document is a cardinality of $D$, i.e., $|D| = N$. ${\cal D} =\{D_1,\dots,D_M\}$ is a collection of $M$ documents that is called a corpus.
\end{definition}

\begin{definition}
A term $t_k$ consist of at least one word, or $t_k = (w_1,\dots,w_l)$, $l\leq k$, $k$ is a number of parameters representing words. $|t_k| = k$ is the number of word of $t_k$, and $l$ is the number of vocabularies in $t_k$.
\end{definition}

Any information available on the web may be obtained with the help of search engines. The search engine is an important part of internet and is one of the easiest and useful tools to research information or to find websites. A search engine allows to categorize and make sense of the information that is available online. The search results generally presented in a list of results and called as \emph{hits}, where information may consist of texts, images, video, hypertext, etc.

\begin{definition}
Let the set of Web pages indexed by search engine be $\Omega$, i.e., a set contains ordered pair of the term $t_{k_i}$ and the web page $\omega_{k_j}$, $(t_{k_i},\omega_{k_j})$, $i=1,\dots,I$, $j=1,\dots,J$. The relation table that consist of two columns $t_k$ and $\Omega_k$, the table is a representation of $(t_{k_i},\omega_{k_j})$, where $\Omega_k = \{(t_k,\omega_k)_{ij}\}\subseteq \Omega$ or $\Omega_k = \{\omega_{k_1},\dots,\omega_{k_j}\}$. The cardinality of $\Omega$ is denoted by $|\Omega|$, and uniform mass probability function is $P:\Omega\rightarrow[0,1]$.
\end{definition} 

\begin{definition}
\label{def:singleton}
Let $t_x$ is a search term, and $t_x\in{\cal S}$ where ${\cal S}$ is set of singleton search term of search engine. A vector space ${\bf x}\subseteq\Omega$ is a singleton search engine event of Web pages that contain an occurrence of $t_x\in\omega_x$, and probability of an event ${\bf x}$ is $P({\bf x}) = |{\bf x}|/|\Omega|\in[0,1]$.
\end{definition}

\begin{proposition}
\label{prop:doubleton}
Let two singleton events $\Omega_x$ and $\Omega_y$ for search terms $t_x$ and $t_y$ respectively. $\Omega_x\cap\Omega_y$ is a doubleton event of $t_x$ and $t_y$ such that $P({\bf x},{\bf y}) = |\Omega_x\cap\Omega_y|/|\Omega|$. 
\end{proposition}
\begin{proof}
By using intersection operator of set to Definition \ref{def:singleton}, we have a direct conclusion, i.e., $P({\bf x}) = |\Omega_x|/|\Omega|$ dan $P({\bf y}) = |\Omega_y|/|\Omega|$ $\Rightarrow$ $
P({\bf x},{\bf y}) = |\Omega_x\cap\Omega_y|/|\Omega|$.
\end{proof}

A search for a particular index term, say $t_x$, it returns a certain number of hits $n_x$, i.e., number of web pages where this term occurred, we obtain $p(t_x) = n_x/|\Omega|$. So for $(t_x,t_y)$ we have a doubleton $n_{xy}$, and $n_x \geq n_{xy}$. $n_{xy}/|\Omega|$ means that $p(t_x|t_y)$ or $p(t_y|t_x)$ are the conditional probabilities. It is clear that a doubleton is a conditional probability of a term for other term.
  
We note that in the conditional probabilities the total number of web pages indexed by search engine, $|\Omega|$ is divided out. Therefore, the conditional probabilities are independent of $|\Omega|$. The conditional probability $P({\bf x}|{\bf y})>0$ means that the search terms $t_x$ and $t_y$ occur together in some web pages or co-occurrence, but also parts of $t_x$ or $t_y$ occur together such that ${\bf x}$ or ${\bf y}$ are in bias.

The edges in a social network refer to ties. A tie relates two actors. Ties could be directed or undirected, and they could be dichotomous (present or absent) or valued (weighted). There may be many types of ties (e.g., kinship, friendship) and the collection of all ties of the same type is a relation. Relations, sometimes called strands, are characterized by content, direction and strength. Let $R$ is a set of relations, the relations among actors formed by sharing attributes, ideas, concepts, etc, between them, which can be depicted as the intersection between their attributes \cite{nasution2011} as follows
\begin{equation}
r_k(a,b)=Z_a\cap Z_b, r_k\in R.
\end{equation}
The content of a relation refers to the source that is exchanged, such as communication about administrative, personal, work-related or social matters. Communication, defined generally as transfer of information or resources, is common among socially related people whereby the electronic trails of communication can be traced include emails \cite{schwartz1993}, newsgroups \cite{agrawal2003}, and instant messaging \cite{resig2004}. In the content, sometime we find self-report, links reported by individual actors. Such links are directed and naturally subjective, such as in classical tools like questionnaires and interviews are based on this principle \cite{wasserman1994}, homepages or profile pages in community-centric sites: LiveJournal weblogs \cite{kumar2004} or Facebook \cite{boyd2004} commonly display a self-professed list of friends within the community. Therefore, a relation can be directed or undirected: one person may give social support to a second person, or there are two relations here: giving support and receiving support. The relations also differ in strength. Such strength can be operationalized in a number of ways \cite{marsden1984,wellman1990}. Therefore, the types of relations important in social network research, it have included the exchange of complex or difficult information \cite{fish1992}, emotional support \cite{fish1992,haythornthwaite1995,rice1987}, uncertain or equivocal communication \cite{daft1986,vandeven1979}, and communication to generate ideas, created consensus \cite{kiesler1992,mccallum2004,mccallum2005,mccallum2007,mcgrath1983,mcgrath1990,mcgrath1991}, support work, forter sociable relations \cite{garton1995,haythornthwaite1996c}, or support virtual community \cite{wellman1997}. In general, the strength relation generated by similarity measures, and some of them are dice, overlap, dan Jaccard \cite{matsuo2006,matsuo2007,nasution2010b} as follows
\begin{equation}
sim(x,y) = \frac{2n_{{\bf x}\wedge{\bf y}}}{n_{\bf x}+n_{\bf y}},
\end{equation}
\begin{equation}
sim(x,y) = \frac{n_{{\bf x}\wedge{\bf y}}}{\min(n_{\bf x},n_{\bf y})},
\end{equation}
and
\begin{equation}
\label{pers:jaccard}
sim(x,y) = \frac{n_{{\bf x}\wedge{\bf y}}}{n_{{\bf x}\vee{\bf y}}}.
\end{equation}
Similarity has its foundation on the sociological idea that friends tend to be alike \cite{carley1991}. Other forms of similarity include having the same communication partners \cite{schwartz1993} and sharing the same opinions or areas of interest \cite{richardson2002}. Each similarity measure relative to the other, and also against the probability, therefore we call these measures as probabilistic relative model. We generally define it as

\begin{definition}
Probabilistic relative model \emph{(PRM)} utilizes the Cartesian product for clustering the nodes in $A$, i.e.,
\[ 
\gamma : A\times A\rightarrow R 
\]
such that $\gamma(a,b)\in R$, $a,b\in A$.
\end{definition}
However, this model not only adriff with a bias in the measurement but also is difficult to produce descriptions of the relationship. In the real world and its application, a social network requires the labels with a weight as a modality as well as explaining the roles of each actor in the social. We define an approach to generate labels as follows

\begin{definition}
\label{definisi:pgm}
Probabilistic generative model \emph{(PGM)} employes a function $\lambda$ for classifying $Z$, i.e., 
\[
\lambda : Z \rightarrow C
\]
such that $\lambda(z) = c$, $z\in Z$, and $c\in C$ is a class of labels, where $C = \{c_1,c_2,\dots,$ $c_{|C|}\}$ is a data set as special target attributes, $|C|\ge 2$ is the number of classes, and $Z\cap C = \emptyset$.
\end{definition}

Statistical natural language processing, an analysis that capture the richness of the language contents of the interactions: the words, the topics, and other high-dimensional specifics of the interactions between actors. Statistically, Bayes theorem has paid dividents in the computing world, especially for artificial intelligence and learning, in which the PGM has played a role in particular. Some PGMs are a direct offspring of Latent Dirichlet Allocation (LDA) \cite{blei2003}, the Multi-label Mixture Model \cite{mccallum1999}, and the Author-Topic Model \cite{steyvers2004,rosenzvi2004}, with the distinction that ART is specifically designed to capture language used in a directed network of correspondents. However, the PGMs concern with extraction of network based on predefined labels only, and thus cannot be adapted to the another description of relation.

\section{The Concept of Probability Use}

The Bayes approach defines the classification problem in terms of probabilities. There are three main concepts required are conditional probability, Bayes theorem, and the bayes decision rule. The conditional probability $(a|D)$, which is used to define independent events of an actor to a document in a corpus, is defined by $P(a|D) = P(a\cup D)/P(D)$, where $P(a|D)$ is the probability that event $a\in A$ happens, given that $D$ is observed. Similarly, $P(D|a) = P(a\cup D)/P(a)$, where $P(D|a)$ is the probability that event $D$ happens, given that $a\in A$ is observed. It then follows (by substitution) that $P(a\cap D)=P(a)P(D|a)$. The premise of Bayes Theorem starts with an initial degree of belief that an event will occur, and then with new information about a degree of belief. These two degrees are reresented, respectively, by the prior probability $P(a|D)$ and the posterior probability $P(D|a)$, which are related by 
\begin{equation}
P(a|D) = \frac{P(a)P(D|a)}{P(D)}
\end{equation}
The Bayes decision rule states that based on the posterior probabilities, it is possible to assign an element $w$ to a class with the largest probability. For example, let $w$ be a data sample (vector of features) and $z_i$ one of the possible classes, then $P(w|z_i)$ is prior probability, because it can be obtained based on prior knowledge.

\begin{proposition}
\label{propos:prob}
Let $P : D \rightarrow [0,1]$ is a mass probability function whereby the probability of $w_i$ is $p(w_i) = 1/|D| \in [0,1]$. If $w_{v_i}$ are the vocabularies in document, then $p(w_{v_i}) = |w_{v_i}|/|D|$, where $|w_{v_i}|$ is number of $w_i$ in $D$. 
\end{proposition}

When working with general web documents in general. We explore the features of documents in a corpus. Thus it can be stated that the probability $P(w)$ is the conditional probability $P(w|D)$ of event $w$ to the document $D$ in the corpus ${\cal D}$. Here are the events of $w$ in the corpus.

\begin{lemma}
\label{lemma:prob}
Let probability of $w_{v_i}$ in $D$ is $p(w_{v_i}) = |w_{v_i}|/|D|$, then probability of $w_{v_i}$ in ${\cal D}$ is
\[
p(w_{v_i}) = \sum_{j=1}^M \frac{|w_{v_i}|}{M|D|} = \sum_{j=1}^M\sum_{k=1}^{|w_{v_i}|} \frac{1}{MN_j}.
\]
Set of $p(w_{v_i})$, ${\rm V} = [p(w_{v_i})|i=1,\dots,n]$ is a vector space where $p(w_{v_i})$ is a weight of $w_i$.
\end{lemma}
\begin{proof}
The direct consequence of Definition \ref{def:doc} and Proposition \ref{propos:prob}.
\end{proof}

Let $t_a$ is a search term, each search engine produces $n_a$ as hit counts, there are as many $n_a$ $ta\in\Omega$, or for all $t_a\in\omega_a$, $|\Omega_a| = n_a$, $t_a\in\Omega_a \subseteq \Omega$. Let us define that $\Omega = \sum_{j=1}^M N_j$, where $M$ is a number of web pages in $\Omega$ and $N_j$ are the number of words/terms in $\omega_j$. Assume that a set of web pages $\Omega$ as corpus ${\cal D}$, then we obtain 
\begin{equation}
P(t_a) = P(a) = |\Omega_a|/|\Omega| = \sum_{j=1}^M\sum_{k=1}^{|t_a|}\frac{1}{MN_j}
\end{equation}
Therefore, we can define that probability of $D\in{\cal D}$ is equivalent to probability of $\omega\in\Omega$.

\begin{proposition}
Let $p$ is a document $D$ from the corpus ${\cal D}$ with probability $P(d)$, then probability $P(\omega) = \sum_{j=1}^M N_j/|\Omega|$
\end{proposition}

\begin{definition}
Let $t_a$ is a term search. ${\cal S}_a = \{S_i|i=1,\dots,n\}$ is a list of snippets that are returned by a search engine for $t_a$, where a snippet $S = \{w_j|j=1,\dots,|S|\}$, and $|S| = \pm 50$ words.
\end{definition}

Each word in document or snippet can take the meaning by giving a vector space. One of methods is to use the $tf\cdot idf$, i.e., 
\begin{equation}
\label{pers:tfidf}
tf\cdot idf = tf(w)\cdot idf(w) = \Big(\sum_{j=1}^M\sum_{i=1}^{|w_{v_i}|}\frac{1}{N_j}\Big)\Big(\log\frac{M}{df(w)}\Big)
\end{equation}
$df(w)$ is the number of document where the word $w$ appearance. Normalization of $tf\cdot idf$ is $tf\cdot idf/h$, $h = $ a highest score of $tf\cdot idf$. 

\begin{proposition}
If $z_k$ is a latent class, the $z_k$ relates to the document $D$.
\end{proposition}

\begin{proof}
Each class of $z_k$ contains a set of words ${\bf w}_k$ as characteritics of class, and each ${\bf w}$ with the values is a vector space whereby if there are $k$ classes of words then intersection of classes is emptyset, $\cap_{i\in I} {\bf w} =\emptyset$. In $D$, based on Lemma \ref{lemma:prob} $\forall w\in D \exists! p(w) \in [0,1]$, so if $\forall w\in\cup_{k=1}^K {\bf w}_k$ then $\exists! p(w) \in [0,1]$. Thus, there is a possibility that $z_k$ associated with $D$.
\end{proof}

Specifically, a document $D$ is potentially related to several topics ${\cal Z}$ with different probabilities, i.e., a set of latent variables ${\cal Z} = \{z_1,\dots,z_K\}$. Therefore, by using same reason  we obtain the following lemma, where latent variables consequently generate a set of words ${\bf w}$.

\begin{lemma}
If $z_k$ is a latent class, the word $w$ can be generated with probability $P(w|z_k)$.
\end{lemma}

\begin{proposition}
If $z_k$ is a latent class, the $z_k$ relates to the actor $a \in A$.
\end{proposition}
\begin{proof}
Each actor will use words to communicate with other actors in the social, some of communication be recorded in the document, so it will relate to the intended actor. When that word also be present in every $z_k$, then we obtain a relation between $a$ and $z_k$.
\end{proof}

The latent variables can generate name appearances $A$ that are closely related to a specific topic. 

\begin{lemma}
If $z_k$ is a latent class, the actor $w$ can generated with probability $P(a|z_k)$.
\end{lemma}

A concept treats document-name-word as a triplet $\langle D,a,w\rangle$ is to represent an instance that a name $a$ appears in document $D$, which contains the word $w$. The relationship inherent in this concept is associated by a set of topics ${\cal Z}$, where a set of latent variables $z$ can break the direct relationships between documents, words, and names.

\begin{theorem}
\label{dalil:campur}
Let ${\cal Z}$ is a set of latent variables, ${\cal Z} = \{z_l|l=1,\dots,L\}$, with size $L$, each of which represents a latent topic, if and only if the relationships between $D\in {\cal D}$, $a\in A$, and $w \in D$ are connected by ${\cal Z}$.  
\end{theorem}
\begin{proof}
A direct result of some of previous lemmas.
\end{proof}

\section{The Framework}

There are five steps of Bayesian methodology \cite{duda1973} to classify something based on Definition \ref{definisi:pgm}: (i) Collect data, and estimate parameters such as mean and coveriance for each class. In this case, assume that all the probability density functions have a Gaussian behavior. (ii) Choose a set of features. (iii) Choose a mode and derive a decision rule with these parameters. (iv) Tray the classifier and apply the decision rule by using a discriminant function, and apply it to a test data set to classify each sample. (v) Evaluate the decision rule. Measure the accuracy/error rate in order to improve the choice of features and the overall design of the classifier. 

In the exploration of social networks that involves web pages, the ties is not just acquired from the hit count, but other features. There are features such as web-snippets that are not only composed a collection of words, a part of them refers to the person names, but there are also the addresses of the web page, URL. For each search term $t_a$, represents a name $a\in A$, the list of snippets ${\cal S}$ will be obtained where $a\in S \subseteq {\cal S}$. Therefore, we can considered the results returned by search engines based on the name as a probability $p(\omega|a)$. Furthermore, each list of snippets can be modeled as a bag of words to generate a vector space in which their weight is obtained by Equation (\ref{pers:tfidf}), and this can be interpreted as the conditional probability, i.e., $P(w|a)$, $P(w|S)$, and $P(w|{\cal S})$ which is meaningful as the current context \cite{nasution2010b} of each person name. 
 
We can provide a latent class $z$ by using the clusters of words from the snippets. For every $a\in A$, ${\cal S}_a$ is potentially a description of the actors and the relations between the actors by using Definition \ref{def:singleton}, Proposition \ref{prop:doubleton} and Equation (\ref{pers:jaccard}) so that the current context is in the trees as the optimal form of relationship between words and person name $a\in A$. In this case, once provided the topics for example based on the existing research group, the trees of words can be selected based on their proximity to the topics by involving singleton and doubleton. The selected tree of words will be a description of the latent variable $z$. Under this scenario, we can derived the aspect model that involves joint probability over $D\times a\times w$ is expressed as mixture
\begin{equation}
P(S,a,w) = P(S)P(a,w|S)
\end{equation}
\begin{equation}
P(a,w|S) = \sum_{z\in {\cal Z}} P(a,w|z)P(z|S)
\end{equation}

The latent variables $z$ connected each instance of which can be obtained from the list of snippets such as names, words and documents, but every instance that was also having relations with one another. Therefore, based on Theorem \ref{dalil:campur}, the optiomal form of this relationship is only connected by a latent variables with disconnecting from each relation between instances, and an symmetric model can be parameterized by
\begin{equation}
P(S,a,w) = \sum_{z\in{\cal Z}} P(z)P(S|z) P(w|z)P(a|z).
\end{equation}

\section{Conclusion and Future Work}
We have proposed a novel framework for acquiring the social networks from snippets. We will demonstrated this after getting the Expectation-Maximization (EM) of the model.

\end{document}